\documentclass[12pt]{amsart}
\usepackage{amssymb,latexsym}
\theoremstyle{plain}
\newtheorem{theorem}{Theorem}

\newtheorem{proposition}[theorem]{Proposition}
\newtheorem{lemma}[theorem]{Lemma}
\theoremstyle{definition}

\newtheorem{remark}[theorem]{Remark}

\newcommand{\PP}{{\mathbb P}}
\newcommand{\CC}{{\mathbb C}}
\newcommand{\RR}{{\mathbb R}}
\newcommand{\NN}{{\mathbb N}}
\newcommand{\sS}{{\mathcal S}}
\newcommand{\sF}{{\mathcal F}}
\newcommand{\sL}{{\mathcal L}}
\newcommand{\rank}{{\rm rank}}

\usepackage{scalefnt}

\title[Hilbert's SOS cones]{Algebraic boundaries of Hilbert's SOS cones}
\author[G.~Blekherman, J.~Hauenstein, J.C.~Ottem, K.~Ranestad, B.~Sturmfels]{
Grigoriy Blekherman, Jonathan Hauenstein, \\ John Christian Ottem,
Kristian Ranestad \\ and Bernd Sturmfels}


\begin{document}

\begin{abstract}
We study the geometry underlying the difference between non-negative
polynomials and sums of squares.
The hypersurfaces that discriminate these two cones for
ternary sextics and quaternary quartics
are shown to be Noether-Lefschetz loci of K3 surfaces.
The projective duals of these hypersurfaces
are defined by rank constraints on Hankel matrices.
We compute their degrees using
numerical algebraic geometry, thereby verifying results
 due to Maulik and Pandharipande. The non-SOS
extreme rays of the two cones of non-negative forms
are parametrized respectively by the  Severi variety of plane rational sextics
and by the variety of quartic symmetroids.
\end{abstract}

\maketitle

\section{Introduction}

A fundamental object in convex algebraic geometry is the
cone $\Sigma_{n,2d}$ of homogeneous polynomials
of degree $2d$ in $\RR[x_1,\ldots,x_n]$ that are sums of squares (SOS).
Hilbert \cite{Hil} showed that  the cones
$\Sigma_{3,6}$ and $\Sigma_{4,4}$ are strictly contained in the
corresponding cones $P_{3,6}$ and $P_{4,4}$ of non-negative polynomials.
Blekherman \cite{Ble}  furnished a geometric
explanation for this containment.
In spite of his recent progress, the geometry of the sets
$P_{3,6} \backslash \Sigma_{3,6}$ and
$P_{4,4} \backslash \Sigma_{4,4}$ remains mysterious.

 We here extend known results
on Hilbert's SOS cones by characterizing their
{\em algebraic boundaries}, that is, the hypersurfaces
that arise as Zariski closures of their topological boundaries.
The algebraic boundary of the cone $P_{n,2d}$ of non-negative polynomials
is the {\em discriminant} \cite{Nie}, and this is also always one
component in the algebraic boundary of $\Sigma_{n,2d}$.
The discriminant has degree  $n(2d-1)^{n-1}$, which equals
$75$ for $\Sigma_{3,6}$ and $108$ for $\Sigma_{4,4}$.
What we are interested in are the other components
in the algebraic boundary of the SOS cones.

\begin{theorem} \label{thm:first}
The algebraic boundary of $\Sigma_{3,6}$ has a unique non-discriminant
component. It has degree $83200$ and consists of forms that are sums of three squares of cubics.
Similarly, the algebraic boundary of $\Sigma_{4,4}$ has a unique non-discriminant
component. It has degree $38475$ and consists of forms that are sums of four squares of quadrics.
Both hypersurfaces define
Noether-Lefschetz divisors in moduli spaces of K3 surfaces.
\end{theorem}

Our characterization of these algebraic boundaries in terms of sums of few squares
is a consequence of \cite[Corollaries 5.3 and 6.5]{Ble}. What is new here is the
connection to K3 surfaces, which elucidates the
  hypersurface of ternary sextics
that are rank three quadrics in cubic forms, and the hypersurface of
quartic forms in $4$ variables that are rank four quadrics in quadratic forms.
Their degrees are coefficients in the modular forms derived by
Maulik and Pandharipande in their paper on
Gromov-Witten and Noether-Lefschetz theory  \cite{MP}.
In Section 2 we explain these concepts and present
the proof of Theorem~\ref{thm:first}.

Section 3 is concerned with the cone dual to $\Sigma_{n,2d}$ and with the
dual varieties to our Noether-Lefschetz hypersurfaces in Theorem \ref{thm:first}.
Each of them is a determinantal variety, defined by rank constraints on
 a $10 \times 10$-Hankel matrix, and it is parametrized by a Grassmannian
via the global residue map in \cite[\S 1.6]{CD}.
We note that Hankel matrices are also known as moment matrices or~ catalecticants.

Section 4 features another appearance of a
Gromov-Witten number \cite{KM} in convex algebraic geometry.
Building on work  of Reznick \cite{Rez}, we shall prove:

\begin{theorem} \label{thm:second}
The Zariski closure of the set of extreme rays of
$P_{3,6} \backslash \Sigma_{3,6}$ is the
Severi variety of rational sextic curves in the projective plane $\PP^2$.
This Severi variety has dimension $17$ and degree $26312976$ in
the $\PP^{27}$ of all sextic curves.
\end{theorem}

We also determine the analogous variety
of extreme rays for quartics in $\PP^3$:

\begin{theorem} \label{thm:third}
The Zariski closure of the set of extreme rays of
$P_{4,4} \backslash \Sigma_{4,4}$ is the
variety of {\em quartic symmetroids}  in $\PP^3$, that is,
the surfaces
whose defining polynomial is the  determinant of a
symmetric $4 \times 4$-matrix of linear forms.
This variety has dimension $24$ in
the $\PP^{34}$ of all quartic surfaces.
\end{theorem}

Section 5 offers an experimental study of the objects in this paper using
 numerical algebraic geometry. We demonstrate
that the degrees $83200$ and $38475$ in Theorem \ref{thm:first} can be found from scratch
using the software {\tt Bertini} \cite{Ber}.
This provides computational validation for the cited results by
Maulik and Pandharipande \cite{MP}.
Motivated by Theorem \ref{thm:third},
we also show how to compute
a symmetric determinantal representation (\ref{eq:symmetroid})
for a given quartic symmetroid.

\smallskip

A question one might ask is: {\em What's the point of integers such as~$38475$?}
One answer is that the exact determination of such degrees signifies an
understanding of deep geometric structures that can be applied to a wider
range of subsequent problems. A famous example is the number $3264$
of plane conics that are tangent to five given conics.
The finding of that particular integer in the 19th century led
to the development of  intersection theory in the 20th century, and
ultimately to numerical algebraic geometry in the 21st century.
To be more specific, our theorems above  furnish novel geometric representations
of boundary sums of squares that are strictly positive,
and of extremal non-negative polynomials that are not sums of squares.
Apart from its intrinsic appeal within algebraic geometry, we expect that our approach,
with its focus on explicit degrees,
will be useful for applications in optimization and~beyond.

\section{Noether-Lefschetz Loci of K3 Surfaces}

Every smooth quartic surface in $\PP^3$ is
a K3 surface. In our study of Hilbert's cone $\Sigma_{4,4}$
we care about  quartic surfaces
containing an elliptic curve of degree $4$.
As we shall see, these are the quartics that are sums of four squares.
K3 surfaces also arise as double covers of
$\PP^2$ ramified along a smooth sextic curve.
In our study of $\Sigma_{3,6}$ we care about  K3 surfaces
whose associated plane sextic is a sum of three squares.
This constraint on K3 surfaces also appeared
in the proof by Colliot-Th\'el\`ene \cite{CT} that  a general sextic in
$\Sigma_{3,6}$ is a sum of four but not three squares of rational functions.

Our point of departure is the Noether-Lefschetz Theorem (cf.~\cite{GH})
which states that a general quartic surface  $S$ in $\PP^3$ has Picard number $1$.
In particular, the classical result by Noether \cite{Noe}
and Lefschetz \cite{Lef} states that every irreducible curve on $S$
is the intersection of $S$ with another surface in $\PP^3$.
This has been extended to general polarized K3 surfaces, i.e.
K3 surfaces $S$ with  an ampel divisor $A$. 
For each even $l \geq 2$,
the moduli space $M_{l}$ of K3 surfaces with a polarization $A$ of degree
$A^{2}=l$ has dimension $19$ and is irreducible. For the general 
surface $S$ in $M_{l}$, the Picard group is generated by $A$.
The locus in the moduli space $M_{l}$ where
the Picard number of $S$ increases to $2$ has codimension one.
We are here interested in one irreducible component of that locus in
$M_2$, and also in  $M_4$.
The relevant enumerative geometry was developed only recently,
by Maulik and Pandharipande \cite{MP}, and our result rests  on theirs.

 \begin{proof}[Proof of Theorem \ref{thm:first}]
 It was shown in \cite{Ble} that   $\partial \Sigma_{3,6}\backslash \partial P_{3,6}$
  consists of ternary sextics $F$ that are sums of three squares over $\RR$.
Over the complex numbers $\CC$, such a sextic $F$ is a  rank three quadric in cubic
forms, so it can be written as
\[ \qquad \qquad
F \,\, = \,\, f h - g^2
\qquad \hbox{where} \,\,f,g,h \in \CC[x_1,x_2,x_3]_3.
\]
Let $S$ be the surface of bidegree $(2,3)$ in $ \PP^{1} {\times} \PP^{2}$
 defined by the polynomial
\[
G \,\, = \,\, f s^{2}+2g st+h t^{2}. \qquad
\]
If $f,g$ and $h$ are general, then the surface $S$ is smooth.
 The canonical divisor on  $\PP^{1}\times \PP^{2}$ has
bidegree $(-2,-3)$, so, by the adjunction formula, $S$ is a K3-surface.  The
projection $S\to \PP^{2}$ is two-to-one,
 ramified along the curve $\{F=0\}\subset \PP^{2}$.
  Up to the actions of $SL(2,\CC)$ and
 $SL(3,\CC)$, there is an
 $18$-dimensional family of surfaces of bidegree $(2,3)$ in
 $\PP^{1}\times \PP^{2}$.  These surfaces determine a divisor
 $D(2,3)$
 in the moduli space $M_{2}$ of K3-surfaces with a polarization of
 degree $2$. The Picard group of a general point $S$
 in this divisor has rank $\leq 2$.

  Let $A\subset S$ be the preimage of a
 general line in $\PP^{2}$, and let $B$ be a general fiber of the projection $S\to
 \PP^{1}$. The classes of the curves $A$ and $B$ are
 independent in the Picard group of $S$.  Therefore the Picard group
 has rank at least $2$ for every surface $S\in D(2,3)$.
  The curves $A$ and $B$  determine the intersection matrix
 \begin{equation}
 \label{eq:inters1}
 \left(
 \begin{array}{cc}
     A^{2} & A\cdot B  \\
     A\cdot B & B^{2}
 \end{array}
 \right) \,\,\, = \,\,\, \left(
 \begin{array}{cc}
     2 & 3  \\
     3 & 0
 \end{array}
 \right).
 \end{equation}

 Conversely, any K3-surface with Picard group generated by classes
 $A$ and $B$ having the
 intersection matrix (\ref{eq:inters1}) has a natural embedding in
 $\PP^{1}\times\PP^{2}$ as a divisor of bidegree $(2,3)$:
 The linear system $|A+B|$ defines an embedding of $S$ into
  the Segre variety $\PP^{1}\times\PP^{2} \subset \PP^5$
  (\cite[Proposition 7.15 and Example 7.19]{SD}).

A general pencil of plane sextic curves contains a finite number
of curves that are ramification loci of
K3 double covers with Picard group of rank $2$ and
 intersection matrix with a given discriminant. In our situation,
   this  number is the degree of the hypersurface that forms the
 Zariski closure of $\Sigma_{3,6} \backslash P_{3,6}$.

We shall derive this number from results of \cite{MP}.
Let $R$ be a general surface of bidegree $(2,6)$ in
$\PP^{1}\times\PP^{2}$, and let $X$ be the double cover of
$\PP^{1}\times\PP^{2}$ ramified along $R$.  The general fiber of the
projection $X\to \PP^{1}$ is a curve $\Pi\cong\PP^{1}$ of K3 surfaces with a
polarization $A$ of degree $l=A^{2}=2$.  Since the surface $R$ has degree two in
the first factor, the curve $\Pi$ defines a conic in the
space of ternary sextics. Section 6 in
\cite{MP} computes the
Noether-Lefschetz number $NL^{\Pi}_{1,3}$ of pairs $(S,B)$ where
$[S]\in \Pi$, and
$B$ is the class of a curve of genus $g(B)=1$ on $S$ and degree
$A\cdot B=3$.  By adjunction, the self-intersection equals $B^{2}=2g(B)-2=0$,
and hence the intersection matrix is as above.

There are two curve classes
on the K3 surface $S$, namely $B$ and $3A-B$, that have
 self-intersection $0$ and intersection number $A\cdot B=A\cdot (3A-B)=3$.
 Therefore each such surface $S$ appears twice in the count of \cite{MP}.
 Also, the curve $\Pi$ is a conic in the space of plane sextic curves. So,
 to the get the count of surfaces $S$ in the Noether-Lefschetz locus
 for a line in the space of sextics, we
 altogether must divide the number $NL^{\Pi}_{1,3}$ by $4$.

 In \cite[Corollary 3]{MP},
 the Noether-Lefschetz number $NL^{\Pi}_{g(B),A\cdot B}$ is expressed as the
 coefficient of the
 monomial $\,q^{\delta}\,$ in the expansion of a modular form
 $\Theta^{\Pi}_{l}$ of  weight $21/2$ as a power series
  in $q^{1/2l}$, where $l=A^{2}$ is the degree of the polarization.
  The exponent of the relevant monomial is $\,
 \delta\,=\,\Delta_{l}(g(B),A\cdot B)/{2l}$,
 where
 $\Delta_{l}(g(B),A\cdot B)$ is the discriminant of intersection matrix
 \begin{small} $\begin{pmatrix}
      l & A\cdot B  \\
     A\cdot B & B^{2}
 \end{pmatrix}$.
 \end{small}
  For the conic $\Pi$ in the space of  sextics,
   the modular form $\Theta^{\Pi}_{2}$ has the expansion
 \[
 \Theta^{\Pi}_{2}=-1+150q+1248q^{\frac{5}{4}}+108600q^{2}+332800q^{\frac{9}{4}}+5113200q^{3}+\ldots .
 \]
In our case, we have $\delta= 9/4$, since $l=2$ and the intersection matrix (\ref{eq:inters1}) has discriminant $9$.
We conclude that the number
 \[
 \frac{1}{4}NL^{\Pi}_{1,3}\,\,= \,\,
\frac{1}{4} 332800 \,\, = \,\,
  83200
 \]
equals the degree of the hypersurface of sextics that are sums of three squares.

\smallskip

We now come to the case of quartic surfaces in $\PP^3$. It was shown in
\cite{Ble} that $\partial \Sigma_{4,4} \backslash \partial P_{4,4}$ consists of
quartic forms $F$ that are sums of four squares over $\RR$. Over the complex numbers $\CC$,
such a quartic $F$ is a rank $4$ quadric in quadrics:
\begin{equation}
\label{eq:K3det}
 F \,\, = \,\, fg-hk \,\,= \,\, {\rm det} \begin{pmatrix}  f & h \\     k & g \end{pmatrix}
\,\,\, \hbox{for some} \, \,f,g,h,k \in \CC[x_1,x_2,x_3,x_4]_2.
\end{equation}
The K3 surface $S$ defined by $F$ contains
two distinct pencils of elliptic curves on $S$, one defined
by the rows and one by the columns of the $2\times 2$ matrix.
Up to the action of $SL(4,\CC)$, the determinantal quartics
(\ref{eq:K3det})  form an $18$-dimensional family, hence a divisor in
the moduli space $M_{4}$. A general surface $S$ in this family has Picard rank $2$,
and its Picard group is generated by the class of a plane section and
the class of an elliptic curve in one of the two  elliptic pencils.

 Conversely, any smooth quartic surface $S$ that contains an
elliptic quartic curve is defined by a determinant $F$ as in (\ref{eq:K3det}).
This form of the equation is therefore characterized by the intersection matrix of $S$.
Let $A$ be the class of
the plane section of $S$ in $\PP^{3}$ and let $B$ and $E$ be the
classes of the curves in the two elliptic pencils. Then  $A$ and $B =
2A-E$ have intersection numbers
\begin{equation}
\label{eq:inters3}  \left(
 \begin{array}{cc}
     A^{2} & A\cdot B  \\
     A\cdot B & B^{2}
 \end{array}
 \right)\,\,= \,\, \left(
 \begin{array}{cc}
     A^{2} & A\cdot E  \\
     A\cdot E & E^{2}
 \end{array}
 \right) \,\, = \,\, \left(
 \begin{array}{cc}
     4 & 4  \\
     4 & 0
 \end{array}
 \right).
\end{equation}
For general $S$, the classes $A$ and $B$ generate the Picard group
and have intersection matrix (\ref{eq:inters3}) with discriminant $\Delta_{4}(1,4) = 16$.
 Let $\Pi$ be a general linear pencil of quartic surfaces in $\PP^{3}$.
 The Noether-Lefschetz number $NL^{\Pi}_{1,4}$ counts pairs
 $(S,B)$ where $[S]\in \Pi$ and $B$ is a curve class on $S$ of degree $4$
 and genus $g(B)=1$.  Since there are two classes of
 such curves on $S$, we get the number of surfaces in the pencil
 containing such a curve class, by dividing $NL^{\Pi}_{1,4}$ by $2$.

 As above, the number $NL^{\Pi}_{g(B),A\cdot B}$ is the
 coefficient  of the monomial $q^{\delta}$ in the expansion  of
 a modular form
 $\Theta^{\Pi}_{l}$ of weight $21/2$ as a power series in $q^{1/2l}$,
 where $l=A^{2}$ is the degree of the polarization.
 Here
 \[
 \delta \,\, = \,\,\frac{\Delta_{4}(g(B),A\cdot B)}{8} \,\, = \,\,\frac{16}{8} \,\, = \,\, 2.
 \]
 The modular form for the
general line $\Pi$ in the space of quartic surfaces equals
  \[
 \Theta^{\Pi}_{4}\,\,=\,\,-1+108q+320q^{\frac{9}{8}}+5016q^{\frac{3}{2}}+76950q^{2}+136512q^{\frac{17}{8}}+\ldots
 \]
 This was shown in \cite[Theorem 2]{MP}. We
 conclude that the degree of the hypersurface of
 homogeneous quartics in $4$ unknowns that are sums of $4$ squares~is
 \[
 \frac{1}{2}NL^{\Pi}_{1,4} \,\, = \,\, \frac{1}{2} 76950 \,\, = \,\,
  38475.
 \]
 This completes the proof of Theorem \ref{thm:first}.
 \end{proof}

\begin{remark}
It was pointed out to us by Giorgio Ottaviani
that the  smooth ternary sextics that are rank three quadrics in cubic forms are known to coincide with the
smooth sextics that have an effective even theta characteristic (cf.~\cite[Proposition 8.4]{Ott}).
Thus the algebraic boundary of Hilbert's SOS cone for ternary sextics is also
related to the theta locus in the moduli space $ \overline{\mathcal{M}}_{10}$.  \qed
\end{remark}

\section{Rank Conditions on Hankel Matrices}

We now consider the convex cone $(\Sigma_{n,2d})^\vee$
 dual to the cone $\Sigma_{n,2d}$. Its elements are the linear forms
$\ell$ on $\RR[x_1,\ldots,x_n]_{2d}$ that are non-negative
on squares. Each such linear form $\ell$ is represented by
its associated quadratic form on $\RR[x_1,\ldots,x_n]_d$,
which is defined by $f \mapsto  \ell(f^2)$.
The symmetric matrix which expresses this quadratic form with respect to
the monomial basis of $\RR[x_1,\ldots,x_n]_d$ is denoted
$H_\ell$, and it is called the {\em Hankel matrix} of $\ell$.
It has format
 $\binom{n+d-1}{d} \times \binom{n+d-1}{d}$, and its rows
 and columns are indexed by elements
 of $\{(i_1,i_2,\ldots,i_n) \in \NN^n : i_1 + i_2 + \cdots + i_n = d \}$.
We shall examine the two cases of interest.

The  Hankel matrix for ternary sextics ($n=d=3$) is the $10 {\times} 10$-matrix
\begin{equation}
\label{eq:Hankel33}
H_\ell \,\, = \,\, \begin{bmatrix}
a_{006} & a_{015} & a_{024} & a_{033} & a_{105} & a_{114} & a_{123} & a_{204} & a_{213} & a_{303} \\
a_{015} & a_{024} & a_{033} & a_{042} & a_{114} & a_{123} & a_{132} & a_{213} & a_{222} & a_{312} \\
a_{024} & a_{033} & a_{042} & a_{051} & a_{123} & a_{132} & a_{141} & a_{222} & a_{231} & a_{321} \\
a_{033} & a_{042} & a_{051} & a_{060} & a_{132} & a_{141} & a_{150} & a_{231} & a_{240} & a_{330} \\
a_{105} & a_{114} & a_{123} & a_{132} & a_{204} & a_{213} & a_{222} & a_{303} & a_{312} & a_{402} \\
a_{114} & a_{123} & a_{132} & a_{141} & a_{213} & a_{222} & a_{231} & a_{312} & a_{321} & a_{411} \\
a_{123} & a_{132} & a_{141} & a_{150} & a_{222} & a_{231} & a_{240} & a_{321} & a_{330} & a_{420} \\
a_{204} & a_{213} & a_{222} & a_{231} & a_{303} & a_{312} & a_{321} & a_{402} & a_{411} & a_{501} \\
a_{213} & a_{222} & a_{231} & a_{240} & a_{312} & a_{321} & a_{330} & a_{411} & a_{420} & a_{510} \\
a_{303} & a_{312} & a_{321} & a_{330} & a_{402} & a_{411} & a_{420} & a_{501} & a_{510} & a_{600}
\end{bmatrix}
\end{equation}
The Hankel matrix for quaternary quartics $(n=4, d=2)$ also has size $10 {\times} 10$:
\begin{equation}
\label{eq:Hankel42}
H_\ell \, = \, \begin{bmatrix}
 a_{0004} &  a_{0013} &  a_{0022} &  a_{0103} &  a_{0112} &  a_{0202} &  a_{1003} &  a_{1012} &  a_{1102} &  a_{2002} \\
 a_{0013} &  a_{0022} &  a_{0031} &  a_{0112} &  a_{0121} &  a_{0211} &  a_{1012} &  a_{1021} &  a_{1111} &  a_{2011} \\
 a_{0022} &  a_{0031} &  a_{0040} &  a_{0121} &  a_{0130} &  a_{0220} &  a_{1021} &  a_{1030} &  a_{1120} &  a_{2020} \\
 a_{0103} &  a_{0112} &  a_{0121} &  a_{0202} &  a_{0211} &  a_{0301} &  a_{1102} &  a_{1111} &  a_{1201} &  a_{2101} \\
 a_{0112} &  a_{0121} &  a_{0130} &  a_{0211} &  a_{0220} &  a_{0310} &  a_{1111} &  a_{1120} &  a_{1210} &  a_{2110} \\
 a_{0202} &  a_{0211} &  a_{0220} &  a_{0301} &  a_{0310} &  a_{0400} &  a_{1201} &  a_{1210} &  a_{1300} &  a_{2200} \\
 a_{1003} &  a_{1012} &  a_{1021} &  a_{1102} &  a_{1111} &  a_{1201} &  a_{2002} &  a_{2011} &  a_{2101} &  a_{3001} \\
 a_{1012} &  a_{1021} &  a_{1030} &  a_{1111} &  a_{1120} &  a_{1210} &  a_{2011} &  a_{2020} &  a_{2110} &  a_{3010} \\
 a_{1102} &  a_{1111} &  a_{1120} &  a_{1201} &  a_{1210} &  a_{1300} &  a_{2101} &  a_{2110} &  a_{2200} &  a_{3100} \\
 a_{2002} &  a_{2011} &  a_{2020} &  a_{2101} &  a_{2110} &  a_{2200} &  a_{3001} &  a_{3010} &  a_{3100} &  a_{4000}
\end{bmatrix} \!\!\!\!
\end{equation}
We note that what we call Hankel matrix is known as
{\em moment matrix} in the literature on optimization and functional analysis,
and it is known as {\em (symmetric) catalecticant} in the literature
on commutative algebra and algebraic geometry.

The dual cone $(\Sigma_{3,3})^\vee$ is the spectrahedron consisting of all
positive semidefinite Hankel matrices (\ref{eq:Hankel33}). The
dual cone $(\Sigma_{4,2})^\vee$ is the spectrahedron consisting of all
positive semidefinite matrices (\ref{eq:Hankel42}).
This convex duality offers a way of representing
our Noether-Lefschetz loci via their
projective dual varieties.

\begin{theorem} \label{thm:fourth}
The Hankel matrices (\ref{eq:Hankel33}) having rank $\leq 7$ constitute a
rational projective variety  of dimension $21$
and degree $2640$. Its dual is the
hypersurface of sums of three squares of cubics.
Likewise, the Hankel matrices (\ref{eq:Hankel42}) having rank $\leq 6$ constitute
 a rational projective variety of dimension $24$ and degree
$28314$.
Its dual is the hypersurface of sums of four squares of quadrics.
\end{theorem}

\begin{proof}
The fact that these varieties are rational and irreducible of the asserted dimensions
can be seen as follows. Consider the Grassmannian ${\rm Gr}(3,10)$ which
parametrizes three-dimensional linear subspaces $F$ of the
$10$-dimensional space
$\RR[x_1,x_2,x_3]_3 $ of ternary cubics.
This Grassmannian is rational and its dimension equals $21$.
The {\em global residue} in $\PP^2$, as defined in \cite[\S 1.6]{CD}, specifies
 a rational map $\,F \mapsto {\rm Res}_{\langle F \rangle}\,$  from
 ${\rm Gr}(3,10)$ into $\PP((\RR[x_1,x_2,x_3]_6)^*) \simeq \PP^{27}$.
 The base locus of this map is the resultant of three ternary cubics, so
  ${\rm Res}_{\langle F \rangle}$ is well-defined whenever the ideal
  $\langle F \rangle$ is a complete intersection in $\RR[x_1,x_2,x_3]$.
  The value ${\rm Res}_{\langle F \rangle}(P)$ of this linear form
  on a ternary sextic $P$ is the image of
    $P$ modulo the ideal $\langle F \rangle$, and it can be computed via
any Gr\"obner basis normal form.
Our map $F \mapsto \ell$ is birational because it has an explicit inverse:
$F = {\rm kernel}(H_\ell)$. The inverse simply maps
the rank $7$ Hankel matrix representing $\ell$ to its kernel.

The situation is entirely analogous for $n=4,d=2$. Here we consider the
$24$-dimensional  Grassmannian ${\rm Gr}(4,10)$ which
parametrizes $4$-dimensional linear subspaces $F$
$\RR[x_1,x_2,x_3,x_4]_2 $.
The global residue in $\PP^3$ specifies a rational map
$$ {\rm Gr}(4,10)\, \dashrightarrow \,\PP((\RR[x_1,x_2,x_3,x_4]_4)^*) \simeq \PP^{34}, \,\,
\,F \,\mapsto \,{\rm Res}_{\langle F \rangle}. $$
This map is birational onto its image, the variety of rank $6$ Hankel matrices
(\ref{eq:Hankel42}), and the inverse of that map takes
a rank $6$ Hankel matrix  (\ref{eq:Hankel42}) to its kernel.

To determine the degrees of our two Hankel determinantal varieties, we argue as follows.
The variety $S_r$ of all symmetric $10 \times 10$-matrices of rank $\leq r$ is
known to be irreducible and arithmetically Cohen-Macaulay,
it has codimension $\binom{11-r}{2}$,
 and its degree is given by the following formula
due to Harris and Tu \cite{HT}:
\begin{equation}
\label{eq:HTbino}
{\rm degree}(S_r) \,\,\, = \,\,\,
 \prod_{j=0}^{9-r} \left( \binom{10+j}{10-r-j} / \binom{ 2j+1}{j} \right).
 \end{equation}
 Thus  $S_r$
 has codimension $6$ and degree  $2640$ for $r=7$, and it
 has codimension $10$ and degree $28314$ for $r= 6$.
The projective linear subspace of Hankel matrices (\ref{eq:Hankel33})
has dimension $27$. Its intersection with $S_7$ was
seen to have dimension $21$. Hence the intersection has the
expected codimension $6$ and is proper. That the intersection is
proper ensures that the degree remains $2640$.
Likewise, the projective linear subspace of Hankel matrices (\ref{eq:Hankel42}) has
dimension $34$, and its intersection with $S_6$ has dimension $24$.
The intersection has the expected codimension $10$, and
we conclude as before that  the degree equals $28314$.

It remains to be seen that the two Hankel determinantal varieties are
projectively dual to the Noether-Lefschetz hypersurfaces in Theorem \ref{thm:first}.
This follows from \cite[Corollary 5.2]{Ble} for sextics curves in $\PP^2$
and from \cite[Corollary 5.7]{Ble} for quartic surfaces in $\PP^3$.
These results characterize the relevant extreme rays of $\Sigma_{3,6}^*$ and
$\Sigma_{4,4}^*$ respectively. These rays are dual to the hyperplanes
that support $\partial \Sigma_{3,6}$ and $\partial \Sigma_{4,4}$ at smooth points
representing strictly positive polynomials.
By passing to the Zariski closures, we conclude that the algebraic boundaries
of $\Sigma_{3,6}\backslash P_{3,6}$ and $\Sigma_{4,4} \backslash P_{4,4}$
are projectively dual to the Hankel determinantal varieties above.
For a general introduction to the relationship between projective duality and
cone duality in convex algebraic geometry we refer to \cite{RS}.
\end{proof}

\begin{remark}
\label{rmk:kristian}
 In the space $\PP({\rm Sym}^2V) $ of quadratic forms on a $10$-dimensional vector space $V^*$, the subvariety
 $S_r $ of forms of rank $\leq r$ is the  dual variety to $S^*_{10-r}\subset
 \PP({\rm Sym}^2V^*)$.  Identifying $V$ with ternary cubics, the space of $10\times 10$ Hankel matrices
 (\ref{eq:Hankel33})  form a $27$-dimensional linear subspace $H 
 \subset \PP({\rm Sym}^2V^{*}) $.
 For $r \leq 3$, we have ${\rm dim}(S_r)<27$ and the variety dual to $H_{10-r}=S^*_{10-r}\cap H$
 coincides with the image $\Sigma_r$  of the birational projection of $S_r$
 into $H^*$.  That image is the variety of sextics that are quadrics of rank $\leq r$ in cubics.  Furthermore, when $r\leq 2$ the projection from $S^*_r$ to $\Sigma_r$ is a morphism, so the degrees of these two varieties coincide.  When $r=3$, the projection is not a morphism and the degree drops to $83200$.
  A similar analysis works for $V = \RR[x_1,x_2,x_3,x_4]_2$ with $r \leq 4$. \qed
\end{remark}

\section{Extreme Non-negative Forms}

For each of Hilbert's two critical cases, in Section 2 we examined
the hypersurface separating $\Sigma_{n,2d}$  and $P_{n,2d}\backslash \Sigma_{n,2d}$.
In this section we take an alternative look at this separation, namely, we focus on the
extreme rays of the cone $P_{n,2d}$ of non-negative forms
 that do not lie in the SOS subcone   $\Sigma_{n,2d}$.
We begin with the following result on zeros of
non-negative forms in the two Hilbert cases.

\begin{proposition}\label{thm CLR}
Let $p$  be a non-negative form in $P_{3,6}$ or $P_{4,4}$. If $p$ has more than $10$ zeros, then
$p$ has infinitely many zeros and it is a sum of squares.
\end{proposition}

\begin{proof}
The statement for $P_{3,6}$ was proved by
Choi, Lam and Reznick in \cite {CLR}.
They also showed the statement for the cone $P_{4,4}$
but with ``$11$ zeros'' instead of ``$10$ zeros''.
To reduce the number from $11$ to $10$, we use
Kharlamov's theorem in \cite{Kha} which states that
the number of connected components
of any quartic surface in real projective $3$-space is $\leq 10$.
See also Rohn's classical work~\cite{Rohn}.
\end{proof}

Recall that a face of a closed convex set $K$ in a finite-dimensional
real vector space is {\em exposed}
if it is the intersection of $K$ with a supporting hyperplane. The extreme rays of $K$
lie in the closure (and hence in the Zariski closure)
of the set of exposed extreme rays  \cite{Sch}.
A polynomial $p \in P_{n,2d}\backslash \Sigma_{n,2d}$  that generates an
extreme exposed ray of $P_{n,2d}$ will be called an
{\em extreme non-negative form}.

Our first goal is to prove Theorem \ref{thm:second},
which characterizes the Zariski closure of the semi-algebraic set of all
extreme non-negative forms for $n=d=3$.

\begin{proof}[Proof of Theorem \ref{thm:second}]
Suppose $p \in P_{3,6} \backslash \Sigma_{3,6}$ is an extreme  form.
By \cite[Lemma 7.1]{Rez}, the polynomial $p$ is irreducible.
Moreover, we claim that
 $|\mathcal{V}_{\mathbb{R}}(p)| \geq 10$. It is not hard to show that $p$ is an extreme non-negative form if and only if $\mathcal{V}_{\mathbb{R}}(p)$
is maximal among all forms in $P_{n,2d}$. In other words, if $p$ is an extreme non-negative form and
$\mathcal{V}_{\mathbb{R}}(p) \subseteq \mathcal{V}_{\mathbb{R}}(q)$ for some $q \in P_{n,2d}$
then $q=\lambda p$ for some $\lambda \in \mathbb{R}$.
Now suppose that $|\mathcal{V}_{\mathbb{R}}(p)| \leq 9$. Then there is a ternary cubic $q \in P_{3,3}$ that vanishes on $\mathcal{V}_{\mathbb{R}}(p)$.
We have $q^2 \in P_{3,6}$ and $\mathcal{V}_{\mathbb{R}}(p) \subseteq \mathcal{V}_{\mathbb{R}}(q)$. This contradicts maximality of $\mathcal{V}_{\mathbb{R}}(p)$. 
By Proposition \ref{thm CLR} we conclude that $|\mathcal{V}_{\mathbb{R}}(p)|=10$.

Let $C$ be the sextic curve in the complex projective plane $\PP^2$ defined by $p = 0$.
Since $C$ is irreducible, it must have non-negative genus. Each point 
in $\mathcal{V}_{\mathbb{R}}(p)$ is a singular point of the complex curve $C$.
As this gives $C$ ten singularities, it follows by the genus formula 
that $C$ can have no more singularities, and furthermore that all of the real zeros of $p$ are ordinary singularities.
 The genus of $C$ is zero and therefore it is an irreducible rational curve.

Let $\mathcal{S}_{6,0}$ denote the Severi variety of rational sextic curves in $\mathbb{P}^2$.
We have shown that $\mathcal{S}_{6,0}$ contains the semi-algebraic set of extreme forms in
$P_{3,6} \backslash \Sigma_{3,6}$.
This is a subvariety in the  $\PP^{27}$ of ternary sextics.
The Severi variety $\mathcal{S}_{6,0}$ is known to be
       irreducible, and the general member $C$ has exactly $10$
           nodes.  Moreover, that set of $10$ nodes in $\PP^2$
    uniquely identifies the rational curve $C$.

    Each rational sextic curve in $\PP^2$ is the image of a
    morphism $\PP^{1} \rightarrow \PP^2$ defined by  three
    binary forms of degree $6$. To choose these, we have
    $3 \cdot 7 = 21$ degrees of freedom. However, the image in $\PP^2$
        is invariant under the natural action of
    the $4$-dimensional  group $GL(2,\CC)$ on the parametrization,
         and hence $\mathcal{S}_{6,0}$ has dimension $21-4 = 17$.
The degree of $\mathcal{S}_{6,0}$ is the number of rational sextics
passing through $17$ given points in $\PP^2$, which is one of the
Gromov-Witten numbers of $\PP^2$.
  For rational curves, these numbers were computed by
Kontsevich and Manin \cite{KM}    using an explicit recursion formula equivalent to the
    WDVV equations.  From their recursion, one gets
   $ {\rm degree}(\mathcal{S}_{6,0}) =   26312976$.

      To complete the proof, it remains to be shown that
     the semi-algebraic set of extreme forms in
$P_{3,6} \backslash \Sigma_{3,6}$ is Zariski dense in the
Severi variety $\mathcal{S}_{6,0}$. We deduce this from \cite[Theorem 4.1 and Section 5]{Rez}.
There, starting with a specific set $\Gamma$ of $8$ points in $\PP^2$, an explicit $1$-parameter family of
non-negative sextics with 10 zeros, 8 of which come from $\Gamma$, was constructed using
{\em Hilbert's Method}.
Furthermore, by Theorem 4.1, Hilbert's Method can be applied to any 8 point configuration in the neighborhood of $\Gamma$.
By a continuity argument, all 8 point configurations sufficiently close to $\Gamma$ will also have a
1-parameter family of non-negative forms with 10 zeros. All such forms are exposed extreme rays.

This identifies a semi-algebraic set of extreme non-negative forms having dimension $16+1 = 17$.
We conclude that this set is Zariski dense in $\mathcal{S}_{6,0}$.
\end{proof}

\begin{remark}
\label{rmk:bruce}
Our analysis implies the following result concerning
$\partial P_{3,6} \backslash \Sigma_{3,6}$.
All exposed extreme rays are sextics with ten acnodes,
and all extreme rays are limits of sextics with ten acnodes.
This proves the second part of Reznick's Conjecture 7.9 in \cite{Rez}.
Indeed, in the second paragraph of the above proof
we saw that $C$ has ten ordinary singularities.
These cannot be cusps since $p \geq 0$. Hence
they have to be what is classically called
{\em acnodes}, or
{\em round zeros} in \cite{Rez}.
\end{remark}

Our next goal is to  derive Theorem \ref{thm:third},
the analogue to Theorem \ref{thm:second} for quartic surfaces in $\PP^3$.
The role of the Severi variety $\mathcal{S}_{6,0}$ is now played by
the variety $\mathcal{QS}$ of quartic symmetroids, {\it i.e.}~the surfaces whose
defining polynomial equals
\begin{equation}
\label{eq:symmetroid}
F(x_1,x_2,x_3,x_4) \quad = \quad {\rm det} \bigl(
A_1 x_1 + A_2 x_2 + A_3 x_3 + A_4 x_4 \bigr),
\end{equation}
where $A_1,A_2,A_3,A_4$ are complex
symmetric $4 \times 4$-matrices.

\begin{lemma}
The variety $\mathcal{QS}$ is irreducible and has codimension $10$
in $\PP^{34}$.
\end{lemma}

\begin{proof}
Each of the four symmetric matrices $A_i$ has $10$ free parameters.
The formula (\ref{eq:symmetroid})
expresses the $35$ coefficients of $F$ as
quartic polynomials in the $40$ parameters, and hence defines a
rational map $\PP^{39} \dashrightarrow \PP^{34}$.
Our  variety $\mathcal{QS}$ is the Zariski closure
of the image of this map, and so it is irreducible.
To compute its dimension, we form the
$35 \times 40$ Jacobian matrix of the parametrization.
By evaluating at a generic point $(A_1,\ldots,A_4)$,
we find that the Jacobian matrix has rank $25$. Hence the dimension of
the symmetroid variety $\mathcal{QS} \subset \PP^{34}$ is $24$.
For a theoretical argument see \cite[page 168, Chapter IX.101]{Jes}.
\end{proof}

A general complex symmetroid $S$ has $10$ nodes,
but not every $10$-nodal quartic in $\PP^3$ is a symmetroid.
To identify symmetroids, we employ the following lemma
 from Jessop's classical treatise \cite{Jes} on singular quartic surfaces.
 Let $S$ be a $10$-nodal quartic  with a node at $p=(0 : 0 : 0 : 1)$.
 Its defining polynomial equals
      $F=f x_{4}^{2} + 2gx_{4}+h$ where $f,g,h \in \CC[x_1,x_2,x_3]$ are
   homogeneous of degrees 2,3,4 respectively. The projection of $S$
   from $p$ is a double cover of the plane with coordinates
   $x_{1},x_{2},x_{3}$ ramified along the sextic curve $C_{p}$
   defined by $g_{}^{2}-f_{}h_{}$.   The curve $C_{p}$ has nodes
   exactly at the image of the nodes on $S$ that are distinct from $p$.
   Since no three nodes on $S$ are collinear, the curve $C_{p}$ has $9$
   nodes in $\PP^2$.
   The following result appears on page 14 in Chapter I.8 of \cite{Jes}.


   \begin{lemma}\label{Lemma:sym} If the sextic
   ramification curve $C_{p}$ is the union of two smooth
       cubics that intersect in $9$ distinct points, then
       the quartic surface $S$
       is a symmetroid and, moreover, the
       ramification curve $C_{q}$ for
       the projection from any node $q$ on $S$ is the union of two  smooth cubic curves.
       \end{lemma}

\smallskip

\begin{proof}[Proof of Theorem \ref{thm:third}]
Let $\mathcal{E}$ denote the semialgebraic set of all non-negative extreme forms $F$
in $ P_{4,4} \backslash \Sigma_{4,4}$.
Each $F \in \mathcal{E}$ satisfies $|\mathcal{V}_\RR(F)| = 10$, by Proposition \ref{thm CLR}
and the same argument as in the first paragraph in the
proof of Theorem \ref{thm:second}.
Thus $\mathcal{E}$ consists of those real
quartic surfaces in $\PP^3$ that have
precisely $10$ real points.

We shall prove that $\mathcal{E}$ is a subset of $\mathcal{QS}$.
Let $F \in \mathcal{E}$ and $S = \mathcal{V}_\CC(F)$ the corresponding complex surface.
Then $S$ is a real quartic with $10$ nodes, and these nodes are real.
Our goal is to show that $S$ is a symmetroid over $\CC$. If $p \in \PP^3_\RR$ is one of the
   nodes of $F$, then the ramification curve $C_{p}$ is a real sextic curve with
   $9$ real nodes at the image of the nodes distinct from $p$.
  Since the nodes on $S$ are the only real points,  these nodes are the only real points on $C_{p}$.

       Through any nine of the nodes of $S$
   there is a real quadratic surface. This quadric is unique;
   otherwise there is a real quadric through all ten nodes    and $F$ is not extreme.
      Let $q$ be a node on $S$ distinct from $p$ and $A$ a real
   quadratic form vanishing on all nodes on $S$ except $q$.
   Consider the pencil of quartic forms
   \[
   F_{t}\,\, = \,\,F+tA^{2}\quad  \hbox{for} \,\, t\in \RR.
   \]
   Suppose $p=(0:0:0:1)$.
   The polynomial $A$ has the form $ u x_4 +v$,
   where $u, v \in \RR[x_1,x_2,x_3]$ have degree $1$ and $2$.
      The equation of $F_t$ is then given by
    \[
   F_t \,\,\, = \,\,\, (f+tu^2) x_4^2\,\, +\,\,2 (g+tuv)  x_4 \,\,+\,\,h+tv^2.
   \]
      Any surface $S_{t}=\{F_{t}=0\}$ has at least $9$ real singular points, namely the
   nodes of $S$ other than $q$.  Since $F$ is non-negative, $F_{t}$ is
    non-negative for $t>0$ with zeros precisely at the $9$ nodes.  On
    the other hand, $F$ has an additional zero at $q$.  Since $A^{2}$
    is positive at $q$, the real surface $\{F_{t}=0\}$ must have a
    $2$-dimensional component when $t<0$.
    Projecting from $p$ we get a pencil of ramification loci
    $C_{p}(t)$.  In the above notation, this family of sextic curves is defined by the forms
    $$ \qquad    G_t \,\,\,= \,\,\, fh- g^2 + t ( h u^2 - 2  g u v +  f v^2) \quad \in \,\, \RR[x_1,x_2,x_3]_6.   $$
    The curves in this pencil have common  nodes at eight real points
    $p_{1},\ldots,p_{8}$ in the plane $\PP^2$, namely the images of the nodes on $S$ other than $p$ and $q$.

Consider the vector space $V$ of real sextic forms that are singular at
$p_{1},\ldots,p_{8}$. Since each $p_i$ imposes $3$ linear conditions,
we have $\dim V \geq 28-3\cdot 8 = 4$. We claim that $\dim V = 4$.
To see this,  consider a general curve $C_p(t)$ with $t>0$. It has only eight
real points, so as a complex curve it is irreducible and smooth outside the eight nodes.
Hence the geometric genus of $C_p(t)$ is 2.
Let $X$ denote the blow-up of
the plane in the points $p_{1},\ldots,p_{8}$, and denote by $C$ the strict
transform of $C_p(t)$ on $X$.
By  Riemann-Roch, $\dim H^0(\mathcal{O}_{X}(C)\big|_C )=3$, since
$C^{2}=4$.
Combined with the cohomology of the exact sequence
\[
0 \to \mathcal{O}_{X} \to \mathcal{O}_{X}(C) \to \mathcal{O}_{X}(C)\big|_C  \to 0,
\]
we conclude that $\dim V=\dim H^0(X,\mathcal{O}_{X}(C) )\le 4$, and hence $\dim V = 4$.

    The pencil $ \RR\{k_{1},k_{2}\}$ of real cubic forms through the eight
    points $p_{1},\ldots,p_{8}$ determine a $3$-dimensional subspace
    $U= \RR\{k_{1}^{2},k_{1}k_{2},k_{2}^{2}\} $ of $ V$, while the sextic forms
    $G_t$ span a $2$-dimensional subspace     $L$ of $ V$.
      Since $G_t$ has no real zeros except the nodes when $t>0$,
   we see that $L$ is not contained in $U$.  Hence $L$ and $U$ intersect
   in a $1$-dimensional subspace of $V$, so there exists a unique value $t_0 \in \RR$ such that
   $C_{p}(t_{0})=K_1\cdot K_2$, where $K_1, K_2 \in \CC[x_1,x_2,x_3]_3$.

We now have   two possibilities: either $K_1$ and
   $K_2$ are both real, or  $K_1$ and $K_2$ are complex conjugates.
   We claim that the latter is the case.
      Consider the intersection $\{G_t=0\} \cap \{K_1\cdot K_2=0\}$. This scheme
      is the union of a scheme of length $32$ supported on the $8$ nodes and a scheme $Z$ of
   length $6\cdot 6-4\cdot 8=4$. Its defining ideal
       $\langle fh-g^2 ,h u^2 -  2g u v +  f v^2 \rangle$ contains the square
      $(gu-2fv)^2$, and thus each of its points has even length.
       Hence $Z$ is either one point of length $4$ or two
   points of length $2$. Since the general $G_t$ does not contain the ninth
   intersection point of $K_1$ and $K_2$, each component of $Z$ is contained in only
   one of the $K_i$. In particular, since $K_i \cap C_p(t)$ contains
   a scheme of length $2$ disjoint from the points $p_{1},\ldots,p_{8}$,
   this shows that $Z$ has two points, one in each of the $K_i$.
   If both $Z_i$ were real then   $K_i \cap Z$ would be real,
   contradicting the fact that $G_t$ has only $8$ real points.
   We conclude that the two cubics
      $K_1,K_2$ are complex conjugates and their only real points
   are the $9$ common intersection points.

   We now claim that $t_0 = 0$. Indeed, if $t_0 < 0$ then
   $S_{t_0}$ has $2$-dimensional real components and
   the real points in the ramification locus $  C_p(t_0) $ would have dimension $1$.
   If $t_0 > 0$ then $S_{t_0}$ has only $9$ real points
   and  $C_p(t_0)$ has only $8$ real points.
   Since $C_p(t_0) = K_1 \cdot K_2$ has $9$ real zeros, it follows that $t_0 = 0$.
   Using Jessop's   Lemma \ref{Lemma:sym}, we now conclude that $F=F_{0}$ is a symmetroid.

We have shown that the semi-algebraic set $\mathcal{E}$ is contained
in the symmetroid variety $\mathcal{QS}$. It remains to be proved that
$\mathcal{E}$ is Zariski dense in $\mathcal{QS}$.
  To see this, we start with any particular extreme quartic.
  For instance, take the following extreme quartic due to Choi, Lam and Reznick
    \cite[Proof of Proposition~4.13]{CLR}:
    \begin{equation}
    \label{eq:CLRquartic} \,\,
    F_{b}\,\,=\,\,\sum_{i,j}x_{i}^{2}x_{j}^{2}\,+\,b\sum_{i,j,k}x_{i}^{2}x_{j}x_{k}\,+ \,(4b^{2}{-}4b{-}2)x_{1}x_{2}x_{3}x_{4}
    \quad \hbox{for} \,\,\, 1 < b < 2,
   \end{equation}
where the sums are taken over all distinct pairs and triples of indices. 
The complex surface defined by $F_b$ has $10$ nodes,
   namely, the points in $\mathcal{V}_\RR(F_b)$.
   Our proof above shows that $F_b$ is a symmetroid.
   Since the Hessian of $F_b$ is positive definite at each of the $10$ real points,
     we   can now perturb these freely in a small
        open neighborhood inside the variety of $10$-tuples
of real points  that are nodes of a symmetroid. Each corresponding
quartic is real, non-negative and extreme. This adaptation of
``Hilbert's method'' constructs
a semi-algebraic family of dimension $24$ in $\mathcal{E}$.
We conclude that
 $\mathcal{QS}$ is the Zariski closure of~$\mathcal{E}$.
\end{proof}

Our proof raises the question whether Lemma \ref{Lemma:sym}
can be turned into an algorithm. To be precise, given an
extreme quartic, such as (\ref{eq:CLRquartic}), what is a practical method for
computing a complex symmetric determinantal representation (\ref{eq:symmetroid})?
We shall address this question in the second half of the next section.

\section{Numerical Algebraic Geometry}

We verified the results of Theorems~\ref{thm:first} and~\ref{thm:fourth}
using the  algorithms  implemented in {\tt Bertini} \cite{Ber}. In what follows
we shall explain our methodology and findings.
An introduction to numerical algebraic
geometry  can be found in \cite{SW}.

The main computational method used in {\tt Bertini} is homotopy continuation.
Given a polynomial system $F$ with the same number of variables
and equations, {\em basic homotopy continuation}
computes a finite set $\sS$ of complex roots of $F$ which contains
the set of isolated roots.  By ``computes $\sS$''  we mean
a numerical approximation of each point in $\sS$ together
with an algorithm for computing each point
in $\sS$ to arbitrary accuracy.  The basic idea is to consider a parameterized
family $\sF$ of polynomial systems which contains $F$.  One first computes
the isolated roots of a sufficiently general member of $\sF$, say $G$,
and then tracks the solution paths starting with the isolated roots of $G$ at $t = 1$
of the homotopy
\[
H(x,t) \,\,= \,\, F(x)(1-t) + tG(x).
\]
The solution paths are tracked numerically using predictor-corrector methods.
For enhanced numerical reliability, the
adaptive step size and adaptive precision path tracking methods of \cite{AMP2} is used.
The endpoints at $t = 0$ of these paths can be computed to arbitrary accuracy using endgames
with the set of finite endpoints being the set $\sS$.
If $F$ has finitely many roots, then $\sS$ is the set of all roots of $F$.
If the variety of $F$ is not zero-dimensional,
 then the set of isolated roots of $F$
is obtained from $\sS$ using the local dimension test of \cite{BHPS}.

Our computations to numerically verify the degrees in
Theorem~\ref{thm:first} only used basic homotopy continuation.
For the $\Sigma_{3,6}$  case, we computed the intersection of the set of rank
three quadrics in cubics with a random line in the space $\PP^{27}$ of ternary sextics.
In particular, for random $p,q \in \CC[x_0,x_1,x_2]_6$,
we computed the complex values of $s$ such that there exists $f,g,h \in \CC[x_0,x_1,x_2]_3$ with
\[
f h - g^2  \,\,= \,\, p + s q.
\]
We used the two degrees of freedom in the parametrization of a rank three
quadric in cubics by taking the
coefficient of $x_0^3$ in $g$ and $x_0^2 x_1$ in $f$ to be zero, and we
 dehomogenized by taking the coefficient of $x_0^3$ in $f$ to be one.
The resulting system $F = 0$
consists of $26$ quadratic and two linear equations in $28$ variables.
Since the solution set of $F = 0$ is invariant under the action of negating $g$,
we considered $F$ as a member of the family $\sF$
of all polynomial systems in $28$ variables consisting of two linear
and $26$ quadratic polynomials which are invariant under this action.
It is easy to verify that a general member of $\sF$ has $2^{26}$
roots, which consist of $2^{25}$ orbits of order 2 under the action of negating $g$.
We took the system $G$ to be a dense linear product polynomial system \cite{LinProd}
with random coefficients which respected this action.
By tracking one path from each of the $2^{25}$ orbits,
which took about 40 hours using 80 processors,
this yielded $166400$ points which correspond to $83200$ distinct values of $s$.

The $\Sigma_{4,4}$ case of Theorem~\ref{thm:first} was solved similarly,
and the number $38475$ was verified.
We took advantage of
 the bi-homogeneous structure of the system
\[ \qquad \qquad
f g - h k \,\,=\,\, p + s q.
\]

Numerical algebraic geometry can be used to compute all irreducible
components of a complex algebraic variety. Here the methods combine
the ability to compute isolated solutions with the use of
random hyperplane sections.  Each irreducible component
$V$ of $F = 0$ is represented by a witness set
which is a triple $(F,\sL,W)$ where $\sL$ is a system of $\dim V$ random linear polynomials
and $W$ is the finite set consisting of the points of intersection of $V$ with $\sL = 0$.

Briefly, the basic approach to compute a witness set for the
irreducible components of $F = 0$ of dimension $k$ is to first compute
the isolated solutions $W$ of $F = \sL_k = 0$ where $\sL_k$ is a system of $k$
random linear polynomials.  The set $W$ is then partitioned into sets,
each of which corresponds to the intersection of $\sL_k = 0$ with an irreducible
component of $F = 0$ of dimension $k$.  The cascade \cite{SV} and regenerative
cascade \cite{HSW} algorithms use a sequence of homotopies to
compute the isolated solutions of $F = \sL_k = 0$ for all relevant
values of $k$.

We applied these techniques to verify the results of Theorem~\ref{thm:fourth}
concerning our $10\times 10$ Hankel matrices.
Our computations combined the regenerative cascade algorithm
with the {\em numerical rank-deficiency method} of \cite{BHPS2}.
In short, if $A(x)$ is an $n\times N$ matrix with polynomial entries, consider the polynomial system
\[
  F_r \,\,= \,\, A(x) \cdot B \cdot \left[\begin{array}{c} I_{N-r} \\ \Xi \end{array}\right]
\]
where $B\in\CC^{N\times N}$ is random, $I_{N-r}$ is the $(N-r)\times (N-r)$ identity matrix, and
$\Xi$ is an $r\times (N-r)$ matrix of unknowns.
One computes the irreducible components of $F_r = 0$
whose general fiber under the projection $(x,\Xi)\mapsto x$
is zero-dimensional. The images of these components are the
components of
\[
\sS_r(A) \,\,= \,\, \{\,x~:~\rank \,A(x) \leq r\}.
\]
The degree of such degeneracy loci is then
   computed using the method of \cite{HS}.

The results on degree and codimension in Theorem~\ref{thm:fourth}
were thus verified, with the workhorse being the regenerative cascade algorithm.
For instance, we ran {\tt Bertini} for $12$ hours on $80$ processors
to find that the variety of Hankel matrices
(\ref{eq:Hankel33}) of rank $\leq 7$ is indeed irreducible
of dimension $21$ and degree $2640$.

\smallskip

We now shift gears and discuss the problem that arose
at the end of Section 4, namely, how to compute a symmetric determinantal representation
(\ref{eq:symmetroid}) for a given extremal quartic $F \in \mathcal{E} \subset \partial P_{4,4} \backslash \Sigma_{4,4}$.
For a concrete example let us consider the Choi-Lam-Reznick quartic in \eqref{eq:CLRquartic}
with $b = 3/2$.
We found that $\,F_{3/2}  =  {\rm det}(M) / \gamma$, where
 $\gamma = -54874315598400(735 \omega + 2201)$, with
$\omega=\frac27\sqrt{-10}$,
 and $M$ is the symmetric matrix with~entries

\begin{equation}\label{eq:symbDet} {\begin{matrix}
m_{11} &  = &        (-11844 \omega+8100) x_1+(3024 \omega+13140) x_3 , \\
m_{12} &  = &       (7980 \omega+14820) x_3  , \\
m_{13} &  = &       (19971 \omega-17460) x_1+(4494 \omega+9600) x_3 , \\
m_{14} &  = &       (-1596 \omega-26790) x_3+(15561 \omega-6840) x_4 , \\
m_{22} &  = &       (30324 \omega-7220) x_2+(20216 \omega+21660) x_3 , \\
m_{23} &  = &        (20216 \omega+21660) x_2+(6384 \omega+27740) x_3 , \\
m_{24} &  = &       (-20216 \omega-21660) x_2-39710 x_3+(7581 \omega-21660) x_4, \\
m_{33} &  = &       (-13230 \omega+31860) x_1+39710 x_2+(-28910 \omega+29910) x_3, \\
m_{34} &  = &       -39710 x_2+(25004 \omega-17100) x_3+((5187/2) \omega-1140) x_4 ,\\
m_{44} &  = &       39710 x_2+(-20216 \omega+37905) x_3+(-30324 \omega+27075) x_4.
\end{matrix}}\end{equation}

A naive approach to obtaining such representations
is to extend the numerical techniques introduced
for quartic curves in \cite[\S 2]{PSV}: after changing coordinates so that
$x_1^4$ appears with coefficient $1$ in $F$, one assumes
that $A_1$ is the identity matrix, $A_2$ an unknown
diagonal matrix, and $A_3$ and $A_4$ arbitrary
symmetric $4 {\times} 4$-matrices with unknown entries.
The total number of unknowns is
$4 {+} 10 {+} 10  = 24$, so it matches the dimension of
the symmetroid variety $\mathcal{QS}$.
With this, the identity (\ref{eq:symmetroid}) translates
into a system of $34$ polynomial equations in
$24$ unknowns.
Solving these equations directly using {\tt Bertini} is currently not possible.
Since the system is overdetermined, {\tt Bertini} actually
uses a random subsystem which has a total degree of $3^6 4^{15}$.  The randomization destroys much of the
underlying structure and solving this system is currently infeasible.

In what follows, we outline a better algorithm based on the underlying geometry of the problem.
The input is a $10$-nodal quartic surface $S=\{F=0\}$.
After changing  coordinates, so that $p=(0{:}0{:}0{:}1)$ is one of 
the nodes, the quartic has the form:
$$ F \,\,= \,\, f x_4^2+2g x_4+h \qquad
 \hbox{where} \,\,\,f,g,h \in \mathbb{R}[x_1,x_2,x_3].
$$
The projection from $p$ defines a double cover $\pi:S\to 
\mathbb{P}^2$ and the ramification locus is the sextic
  curve whose defining polynomial is $fh-g^{2}$ and splits into a product of two
  complex conjugate cubic forms $K_1, K_2$.
  The intersection of $S$ with $\{K_1=0\}$, regarded as a cubic cone in $\PP^3$,
     is supported on the branch locus of the double cover
      and therefore equals two times a curve $C$ of degree 6. The curve $C$ has a triple point at the vertex $p$,
       its arithmetic genus is $ 3$, and it is arithmetically Cohen-Macaulay.
By the Hilbert-Burch Theorem,
the ideal of $C$ is generated by the $3\times 3$-minors $g_1,\ldots, g_4$ of a $3\times 4$ matrix
whose entries are linear forms in $\mathbb{C}[x_1,x_2,x_3,x_4]_1$:
\begin{equation}\label{eq:HilbertBurch}
\begin{bmatrix}
l_{11} & l_{12} & l_{13} & l_{14}  \\
l_{21} & l_{22} & l_{23} & l_{24}  \\
l_{31} & l_{32} & l_{33} & l_{34}
\end{bmatrix}
\end{equation}
The rows of this matrix give
  three linear syzygies between the four cubics $g_i$. Furthermore, $F$ itself is in the ideal generated by these cubics, and so there is a linear relation $F=l_1g_1+\cdots+l_4g_4$.
  Hence the quartic $F$ is equal, up to multiplication by a non-zero scalar in $\mathbb{C}$,
  to the determinant of the matrix
\[ L \quad = \quad
\begin{bmatrix}
l_{1} & -l_{2} & l_{3} & -l_{4}  \\
l_{11} & l_{12} & l_{13} & l_{14}  \\
l_{21} & l_{22} & l_{23} & l_{24}  \\
l_{31} & l_{32} & l_{33} & l_{34}
\end{bmatrix}.
\]
To find a \emph{symmetric} matrix $M$
with the same property, we solve the linear system $PL = (PL)^T$
for some matrix $P \in {\rm GL}(4,\CC)$ and define $M = PL$.

A numerical version of the above algorithm is almost exactly as explained above
except that a basis for the ideal $I_C$ of the genus $3$ curve $C$
is found by computing a large sample of points in the intersection
$\{F=K_1=0\}$, and then computing a basis $g_1,\ldots,g_4$ for the $4$-dimensional 
space of cubic forms vanishing on this set.
Next, a basis for the $3$-dimensional set of linear 
syzygies between these cubics is computed. This
yields the matrix in \eqref{eq:HilbertBurch} whose $3\times 3$ minors 
are the four cubics $g_i$.
For the quartic \eqref{eq:CLRquartic} with $b = 3/2$, we used {\tt Bertini} to compute 100 random
points in this intersection and then used standard numerical linear algebra algorithms.
In total, it took $30$ seconds to compute a symmetric determinantal representation for $F_{3/2}$.
To four digits, with  $i = \sqrt{-1}$, the output we found is
 the symmetric matrix $M$ with entries

\begin{equation}\label{eq:numDet}{\scalefont{.54} \begin{matrix}
m_{11} &=& (15.5378 +  5.6547i)x_1 - (20.4008 -  5.8116i)x_2 - (23.1956 + 16.9236i)x_3 + (12.4987 + 26.8206i)x_4 , \\
m_{12} &=& (18.3458 -  5.8125i)x_1 - (14.0867 - 25.1505i)x_2 - (35.0029 -  5.2948i)x_3 + (36.1417 + 15.7167i)x_4 , \\
m_{13} &=& (11.6232 +  5.6624i)x_1 - (15.6076 -  5.9393i)x_2 - (17.3057 + 12.3685i)x_3 + (11.0079 + 22.8305i)x_4 , \\
m_{14} &=& (25.7222 +  1.2098i)x_1 - (27.4233 - 22.3864i)x_2 - (45.8046 + 14.1068i)x_3 + (35.3836 + 37.8454i)x_4 , \\
m_{22} &=& (12.6315 - 18.4638i)x_1 + ( 9.6932 + 37.6953i)x_2 - (26.0269 - 34.9909i)x_3 + (49.9098 - 16.2993i)x_4 , \\
m_{23} &=& (14.6285 -  3.0705i)x_1 - ( 9.5983 - 20.6203i)x_2 - (25.8489 -  4.2265i)x_3 + (31.1616 + 13.0794i)x_4 , \\
m_{24} &=& (24.1544 - 17.3589i)x_1 - ( 5.2755 - 47.6528i)x_2 - (52.3363 - 27.5281i)x_3 + (68.7313 +  6.4353i)x_4 , \\
m_{33} &=& (~8.5030 +  5.3275i)x_1 - (11.9127 -  5.6822i)x_2 - (12.9473 +  8.9555i)x_3 + (9.6646 +  19.4288i)x_4 , \\
m_{34} &=& (19.6130 +  2.9165i)x_1 - (20.0754 - 19.3371i)x_2 - (34.2042 +  9.9911i)x_3 + (30.7454 + 32.0581i)x_4 , \\
m_{44} &=& (37.6831 - 10.7034i)x_1 - (27.3051 - 52.2852i)x_2 - (80.4558 -  2.6947i)x_3 + (79.5452 + 43.7001i)x_4 .
\end{matrix}}\end{equation}

The symbolic solution \eqref{eq:symbDet} 
and the numerical solution \eqref{eq:numDet} are in the 
same equivalence class of symmetric matrix representations.  In fact,  
we close with the result that the output of the algorithm is 
 essentially unique, independant of the choice of node $p$ and cubic 
 form $K_{i}$:

\begin{proposition} \label{prop:conjugation}
For any $10$-nodal symmetroid $F \in \mathcal{QS}$,
the representation (\ref{eq:symmetroid})
is unique up to the natural action  of $\,{\rm GL}(4,\CC)$ via $A_i \mapsto U A_i U^T$ for $i=1,2,3,4$.
 \end{proposition}

\begin{proof}
Let $M = \sum x_i A_i$ be a symmetric matrix such that $F={\rm det}( M)$. Any
three of the four rows of $M$ determine a curve $C$ by taking $3\times 3$ minors.
This gives a $4$-dimensional linear system $L_M$ of curves of arithmetic genus $3$ and degree $6$ on
$S$.  The doubling of any curve in $L_M$ is the complete intersection of $S$ and a cubic surface defined by a $3\times3$-symmetric
 submatrix of $M$. Conversely, the linear system determines the matrix $M$ up to a change of basis.

 Each curve in $L_M$ passes through all the nodes of $S$, and
 these are the common zeros of the curves in $L_M$.
 If $\tilde{S}$ is the smooth K3 surface obtained by resolving the nodes, then by Riemann-Roch, 
 $L_M$ defines a \emph{complete} linear system on $\tilde{S}$. Since $\mbox{Pic}(\tilde{S})$ is 
 torsion-free, we see that $L_M$ is uniquely determined as the linear system of degree $6$ curves on $S$ passing through all nodes and whose doubling form a complete intersection. Therefore the equivalence class of the symmetric matrix representation
 is also unique.
\end{proof}


\bigskip

\noindent {\bf Acknowledgments.}
This project was started at the Mittag-Leffler Institute,
Djursholm, Sweden, whose support and
hospitality was enjoyed by all authors. We thank Paul Larsen, Giorgio Ottaviani, 
Rahul Pandhariphande and
Ulf Persson for helpful discussions and comments.
GB, JH and BS
were also supported by the US National Science Foundation.

\bigskip
\medskip

\noindent
Grigoriy Blekherman, Georgia Inst.~of Technology, Atlanta, USA,
{\tt grrigg@gmail.com}

\noindent
Jonathan Hauenstein, Texas A\&M, College Station, USA,
{\tt jhauenst@math.tamu.edu}

\noindent
John Christian Ottem, University of Cambridge, England,
{\tt jco28@dpmms.cam.ac.uk}

\noindent
Kristian Ranestad, University of Oslo, Norway,
{\tt ranestad@math.uio.no}

\noindent
Bernd Sturmfels, UC Berkeley, USA,
{\tt bernd@math.berkeley.edu}

\end{document}